\documentclass[12pt]{amsart}

\usepackage[utf8]{inputenc}
\usepackage[T1]{fontenc}
\usepackage{graphicx}
\usepackage{mathtools, bm}
\usepackage{amssymb, bm}
\usepackage{amsthm}
\usepackage{amsfonts}
\usepackage{amsmath}
\usepackage{stmaryrd}
\usepackage{hyperref, enumerate}
\usepackage{setspace}
\usepackage{dsfont}
\usepackage{array, color}
\usepackage{fancybox}
\usepackage{marvosym}
\usepackage{wasysym}
\usepackage{soul}
\usepackage{tikz-cd}
\usepackage{hhline, bm}
\usepackage{mathrsfs}
\usepackage{enumitem}
\usepackage{comment}
\usepackage{leftidx}

\theoremstyle{definition}\newtheorem{definition}{Definition}[section]
\newtheorem{theorem}{Theorem}[section]
\newtheorem{proposition}[theorem]{Proposition}

\newtheorem{corollary}[theorem]{Corollary}
\newtheorem{remark}[theorem]{Remark}

\newcommand{\function}[5]{\begin{array}{lrcl}
 #1 : & #2 & \longrightarrow & #3 \\
    & #4 & \longmapsto & #5 \end{array}}
\newcommand{\Z}{\mathbb{Z}}

\newcommand{\R}{\mathbb{R}}
\newcommand{\som}[3]{\underset{#1}{\overset{#2}{\sum}} #3}

\newcommand{\free}[1]{[ #1 ]}
\newcommand{\scalar}{\langle \cdot ,\cdot \rangle}
\title{A Combinatorial Approach to Voronoï Theory}
\author{Romain Ménabé}
\address{Romain Ménabé, Univ. Grenoble Alpes, CNRS, IF, 38000 Grenoble, France. \textit{E-mail adress :} \href{mailto:romain.menabe@univ-grenoble-alpes.fr}
{\texttt{romain.menabe@univ-grenoble-alpes.fr}}
}
\begin{document}
\maketitle
\begin{abstract}
This article introduces crystallographic wraiths, new combinatorial invariants attached to Euclidean lattices and linked to the geometry of Voronoï cells. They are based on geometric properties of shortest linear dependencies among relevant Voronoï vectors.\\
Moreover, it explores the new theory of crystallographic wraiths, proves that only finitely many occur in each dimension, and extends the construction to a larger class of combinatorial objects. As an application, it provides a new combinatorial proof of Voronoï's finiteness theorem for perfect lattices\footnotemark. 
\end{abstract}

\section{Introduction}
A Euclidean lattice is a free abelian group $\Lambda$ of finite rank, together with a positive definite symmetric bilinear form $\scalar$ on $\Lambda \otimes_\Z \R$.\\
For the rest of this article, let $(\Lambda,\scalar)$ be a Euclidean lattice, and denote by $\mathbb{E}=\Lambda \otimes_{\Z} \R$ the associated Euclidean space of dimension $d = rk(\Lambda)$. We denote by $Q$ the quadratic form associated to $\scalar$.\\
Most of the information about the geometry of a lattice is given by a domain associated with the work of Voronoï in \cite{Vor1}.\\
Given a lattice vector $g \in \Lambda$, one defines the \emph{Voronoï cell associated with $g$} as$$
V_{\Lambda}(g)=\left\{x\in \mathbb{E}\;\middle|\;Q(x-g)\leq Q(x-h)\text{ for all } h \in \Lambda\right\}.$$
Thus, $V_{\Lambda}(g)$ is the set of all elements of $E$ closer to $g$ than to any other lattice vector.\\
Since two Voronoï cells differ by a lattice vector, we will focus on the cell at the origin, given by 

$$V_{\Lambda}=V_{\Lambda}(0)=\left\{x\in \mathbb{E}\;\middle|\;\langle x,g \rangle \leq \frac{1}{2} \langle g,g \rangle \text{ for all } g \in \Lambda \right\}.$$
The polytope $V_{\Lambda}$ is a fundamental domain for $\Lambda$ acting by translations on $\mathbb{E}$.
A lattice vector $f \in \Lambda$ is called a \emph{facet vector} if the hyperplane 
$$H_f=\left\{x\in \mathbb{E}\;\middle|\;\langle x,f \rangle=\frac{1}{2} \langle f,f \rangle \right\}$$
contains a facet of the Voronoï cell\footnote{Several authors, such as Conway and Sloane, often call these vectors \emph{relevant Voronoï vectors} or \emph{strict Voronoï vectors} (see for example \cite{ConSloBook})}.\\
Since $V_{\Lambda}=-V_{\Lambda}$ is centrally symmetric, facet vectors come in opposite pairs.\\ Thus, we denote by $\mathcal{F}_{\Lambda}$ a set of representatives of all pairs, and therefore $\pm \mathcal{F}_{\Lambda}$ represents the set of all facet vectors.\\
Facet vectors are canonical generators and contain essential information about the combinatorics of the lattice tiling $\mathcal{V}_{\Lambda} + \Lambda$. They are at the core of this work.\\
The following theorem, due to Voronoï and later introduced to lattice theory by B.B. Venkov, characterizes facet vectors.

\begin{theorem}[\cite{Vor1} p.277,\cite{Ven1},\cite{ConSloBook} p.477]
\label{Venkov}
~\\A nonzero vector $ f \in \Lambda $ is a facet vector if and only if $\pm f$ are the only minimal norm vectors in the class $f+2\Lambda$.
\end{theorem}
Since $\mathbb{F}_2^*$ is trivial, projective spaces over $\mathbb{F}_2$ can be identified with sets of nonzero elements of $\mathbb{F}_2$-vector spaces.
Theorem \ref{Venkov} connects the geometry of the Voronoï cell to the projective space $\Lambda/2\Lambda \backslash\{0\}$ of dimension $d-1$ over the field $\mathbb{F}_2$ of two elements.\\
In Section \ref{S2}, we explore this connection and use it to introduce combinatorial objects associated to Euclidean lattices, encoding shortest linear dependencies between facet vectors related to a subset of lines in the projective space $\Lambda/2\Lambda\backslash\{0\}$. We call them crystallographic wraiths.\\
In Section \ref{S3}, we generalize crystallographic wraiths further, building on the results obtained in Section \ref{S2}.\\
In Section \ref{S4}, we give a new proof of Voronoï's finiteness result for perfect lattices based on wraiths.

\section{Crystallographic wraiths}
\label{S2}
The goal of this section is to construct crystallographic wraiths, combinatorial invariants associated to Euclidean lattices. There are finitely many of them in a given dimension.

The motivation for the work developed below comes from a fact demonstrated later in this section, stating that the shortest dependencies between relevant Voronoï vectors generate the kernel of a natural morphism.\\

We fix a set $\mathcal{F}=\mathcal{F}_{\Lambda}$ of representatives for pairs of facet vectors. Elements of $\mathcal{F}$ are called \emph{favors} (\textbf{FA}cet \textbf{V}ectors with \textbf{OR}ientation). The choice of $\mathcal{F}$ matters for the construction only up to a notion of isomorphism.\\

Given a set $\mathcal{E}$, we denote by $\Z[\mathcal{E}]$ the free abelian group over $\mathcal{E}$.\\
Then, we denote by $\Z\free{\mathcal{F}}$ the standard Euclidean lattice with orthonormal basis $([f])_{f \in \mathcal{F}}$ given by a set of favors representing all pairs of opposite facet vectors. Therefore, every element of $\Z\free{\mathcal{F}}$ has the form $\som{f \in \mathcal{F}}{}a_f [f]$.\\

There is a natural morphism $\pi=\pi_{\mathcal{F}}$ from $\Z \free{\mathcal{F}}$ to $\Lambda$ as follows :
$$\function{\pi}{\Z\free{\mathcal{F}}}{\Lambda}{\som{f \in \mathcal{F}}{}a_f [f]}{\som{f \in \mathcal{F}}{}a_f f}.$$

Since Voronoï cells tile $\mathbb{E}$, facet vectors generate $\Lambda$ and $\pi$ is surjective. \\

We now study the kernel of this map. One might see it as the set of linear dependencies among facet vectors. Within those dependencies, we will focus on the simplest nontrivial ones.
\begin{definition}
Let $f_k, f_l$ and $f_m$ three elements of $\mathcal{F}$ and $\epsilon_k, \epsilon_l$ and $\epsilon_m$ in $\{\pm 1\}$. An element of the form
$$\varepsilon_{k}[f_{k}]+\varepsilon_{\ell}[f_{\ell}]+\varepsilon_{m}[f_{m}]$$
in $\ker(\pi)$ is called a \textit{line relation}.
\end{definition}
Just as facet vectors, line relations come in opposite pairs.\\
We denote by $\mathcal{L}_{\mathcal{F}}$ a set of representatives for pairs of opposite line relations. Hence, $\pm \mathcal{L}_{\mathcal{F}}$ is the set of line relations.\\
There is a natural map from the free abelian group generated by line relations $\Z \free{ \mathcal{L}_{\mathcal{F}}}$ into $\Z \free{\mathcal{F}}$.\\

Applying Theorem \ref{Venkov}, we obtain a bijection between $\mathcal{F}$ and a subset $\overline{\mathcal{F}}$ of the projective space $(\Lambda/2\Lambda) \backslash \{0\}$.\\
This projective point of view is intimately linked with the viewpoint used by Conway and Sloane in their algorithm of reduction of three dimensional lattices (\cite{ConSlo1}). However, unlike them, we only consider the subset of $\Lambda/2\Lambda\backslash \{0\}$ corresponding to facet vectors.\\
Moreover, the injectivity of the restriction to $\mathcal{F}$ of the projection $\overline{\pi} : \Z\free{\mathcal{F}} \longrightarrow \Lambda/2\Lambda$ implies that $\ker(\pi)$ has no element of norm 1 or 2.\\
Two line relations are therefore equal up to sign or involve at most one common basis element.\\
Furthermore, each line relation of $\mathcal{L}_{\mathcal{F}}$ gives rise to three elements $f_k$, $f_l$ and $f_m$ associated to a line (three aligned points) in the projective space $\Lambda/2\Lambda\backslash \{0\}$. This last set will be denoted $\mathcal{L}_{\overline{\mathcal{F}}}$, and its elements will be called the \emph{kernel lines}.\\
Since line relations are minimal norm vectors in $\Z\free{\mathcal{F}}$, there is a one-to-one correspondance between the set $\mathcal{L}_{\mathcal{F}}$ of line relations and the set $\mathcal{L}_{\overline{\mathcal{F}}}$ of kernel lines.

The following theorem shows the relevance of line relations.

\begin{theorem}
\label{exseq}
The sequence
$$\Z \free{\mathcal{L}_{\mathcal{F}}} \overset{\iota}{\longrightarrow} \Z \free{\mathcal{F}} \overset{\pi}{\longrightarrow} \Lambda \longrightarrow 0$$ is exact.\\
In particular, the natural homomorphism $\iota : \Z\free{\mathcal{L}_{\mathcal{F}}} \longrightarrow \Z\free{\mathcal{F}}$ surjects onto $\ker(\pi)$ and $\ker(\pi)=\langle \mathcal{L}_{\mathcal{F}} \rangle$ is the subgroup of $\Z\free{\mathcal{F}}$ generated by line relations.
\end{theorem}
\textit{Sketch of proof :}\\
The essential part here is to prove that $\ker(\pi)$ is generated by line relations.\\
After ordering its summands, an element in $\ker(\pi)$ can be seen as a piecewise linear loop in $\mathbb{E}=\Lambda \otimes_Z \R$. We fill this loop with 2-simplices (triangles) obtaining essentially a topological disc (up to self-intersection).\\
Thanks to the General Position Theorem (see \cite{RS}), a generic perturbation of this disc does not intersect faces of codimension greater than 2 of the Voronoï tesselation. Then, using a Theorem in \cite{MAG}, a $(d-2)$-face is the intersection of either 3 or 4 Voronoï cells. The intersection of three Voronoï cells corresponds to a line relation, while the intersection of four Voronoï cells can be described as a "rectangle", encoding two pairs of orthogonal facet vectors. The orientation of the space gives us a decomposition into line relations of each triangle, and thus of the element of the kernel. Intersections of four Voronoï cellscorrespond to commutation-relations and can be neglected.

This theorem admits an interpretation in terms of geometric group theory. One associates to a Euclidean lattice the presentation of a free abelian group with generators $\pm \mathcal{F}$ and relations $\mathcal{L}_{\mathcal{F}}$. With this point of view, one can remark that, during the proof, we construct the Van Kampen diagram of words corresponding to elements in the kernel.\\

In this proof, we used the fact that if a $(d-2)$-face of a cell is the intersection of exactly three Voronoï cells, then there exists a line relation between the facet vectors linking these Voronoï cells. However, not all line relations occur in this way. Thus, given a lattice, some line relations may never appear in Van Kampen diagrams as constructed above. Such examples exist in dimension greater than 4.\\

We notice that, using the first isomorphism theorem on $\pi$, we have that $\Z\free{\mathcal{F}}/\langle \mathcal{L}_{\mathcal{F}} \rangle$ is isomorphic to $\Lambda$. Hence, we can recover the projective structure of $\Lambda/2\Lambda\backslash \{0\}$.\\

The pair $(\mathcal{F},\mathcal{L}_{\mathcal{F}})$ therefore defines an invariant encoding a part of the combinatorial information of the tiling induced by $\Lambda$. However, the choice of signs for the elements of the set $\mathcal{F}$ of favors makes it not entirely canonical.\\

We deal with this problem by introducing a notion of isomorphism accounting for sign choices.

Let us start by recalling that morphisms of free-modules are completely determined by the images of the source generators. Moreover, if $U$ is a subset of $S$ (respectively $T$), then it corresponds to the subset of $\Z\free{S}$ (respectively $\Z\free{T}$) where each vector has only one non-zero coordinate, equal to 1, and associated to an element of $U$. By abuse of notation, we also denote by $U$ this subset of $\Z\free{S}$ (respectively $\Z\free{T}$).\\

Let now $\Lambda_1$ and $\Lambda_2$ be two lattices with chosen sets $\mathcal{F}_1$ and $\mathcal{F}_2$ of favors.\\
We say that the pairs $(\mathcal{F}_1, \mathcal{L}_{\mathcal{F}_1})$ and $(\mathcal{F}_2, \mathcal{L}_{\mathcal{F}_2})$ are \emph{isomorphic} if there exists an isomorphism $\rho : \Z\free{\mathcal{F}_1} \longrightarrow \Z\free{\mathcal{F}_2}$ such that  $\rho(\pm \mathcal{F}_1)=\pm \mathcal{F}_2$ and $\rho(\pm \mathcal{L}_{\mathcal{F}_1})=\pm \mathcal{L}_{\mathcal{F}_2}$.\\

$\rho$ induces a projective bijection between $\Lambda_1/2\Lambda_1 \backslash \{0\}$ and $\Lambda_2/2\Lambda_2 \backslash \{0\}$ sending the points $\overline{\pi_{\mathcal{F}_1}}(\mathcal{F}_1)$ bijectively onto the points $\overline{\pi_{\mathcal{F}_2}}(\mathcal{F}_2)$ and the lines $\overline{\pi_{\mathcal{F}_1}}(\mathcal{L}_{\mathcal{F}_1})$ bijectively onto the lines $\overline{\pi_{\mathcal{F}_2}}(\mathcal{L}_{\mathcal{F}_2})$.\\
In other words, one would say that such an isomorphism $\rho$ establishes a bijection between the projective structures associated to the lattices.

One can also immediately notice that the choice of two different sets of favors at the beginning of the construction will result in the same class for this relation.\\

This leads to the following definition.
\begin{definition}
\emph{The $d$-dimensional crystallographic wraith associated to $\Lambda$} is the isomorphism class of any pair of the form ($\mathcal{F}$, $\mathcal{L}_{\mathcal{F}}$) where $\mathcal{F}$ is a set of favors of $\Lambda$.
\end{definition}

It also follows immediately from the definitions and the construction above that two lattices with the same combinatorial Voronoï cell have the same crystallographic wraith. An example in dimension 4 proves that the converse of this fact is not true.\\
To summarize, the subdivision of space obtained through combinatorics of Voronoï cells is finer than the one obtained with crystallographic wraiths.

One can remark that this subdivision of space was also obtained with different tools by Ryshkov and Baranovskii in \cite{RandB}.
However, they did not remark the analogies with the geometry of projective spaces and the construction through line relations leading to the notion of crystallographic wraith.\\

Euclidean lattices admit a natural notion of decomposability. More precisely, a Euclidean lattice is decomposable if it is the orthogonal sum of two nontrivial Euclidean sublattices.\\
We extend this notion to crystallographic wraiths. 
\begin{proposition}
If $(\mathcal{F}_1,\mathcal{L}_{\mathcal{F}_1})$ is a $d_1$-dimensional crystallographic wraith and $(\mathcal{F}_2,\mathcal{L}_{\mathcal{F}_2})$ is a $d_2$-dimensional crystallographic wraith, then $(\mathcal{F}_1 \cup \mathcal{F}_2,\mathcal{L}_{\mathcal{F}_1} \cup \mathcal{L}_{\mathcal{F}_2})$ is a $(d_1+d_2)$-dimensional crystallographic wraith.
A wraith of this form is called \emph{decomposable}.
\end{proposition}
The interest of this extension of decomposability is the following.
\begin{proposition}
The decomposability of crystallographic wraiths corresponds to the decomposability of the underlying Euclidean lattices.
\end{proposition}

Before proceeding on the application to perfect lattices, we introduce, in the next section, a generalization of crystallographic wraiths.
\section{Wraiths}
\label{S3}
The goal of this section is to present a generalization of crystallographic wraiths no longer linked to Euclidean lattices.
\begin{definition}
\label{def1}
A \emph{$d$-dimensional wraith} is a pair $(\Z^n,\mathcal{R})$ where $\Z^n$ is the standard Euclidean lattice with the orthonormal basis $\{e_1,...,e_n\}$ and $\mathcal{R} \subset \Z^n$ is a subgroup with the following properties :\begin{enumerate}
\item $\mathcal{R}$ does not contain elements of (squared Euclidean) norm 1 and 2,
\item $\mathcal{R}$ is generated by all elements of (squared Euclidean) norm 3 in $(\mathcal{R} \otimes_\Z \mathbb{Q}) \cap \Z^n$,
\item The quotient group $\Gamma = \Z^n/\mathcal{R}$ is isomorphic to $\Z^d \times T$ where $T$ is a torsion group of odd order.
\end{enumerate}
\end{definition} 

\begin{remark}
There is a slightly weaker definition replacing the condition (2) by (2') : $\mathcal{R}$ is generated by elements of (squared Euclidean) norm 3 in $(\mathcal{R}\otimes_\Z \mathbb{Q}) \cap \Z^n$.\\
The stronger version of Definition \ref{def1} minimizes the size of the torsion group $T$ and ensures that $\mathcal{R}$ is determined by its linear span $(\mathcal{R} \otimes_{\Z} \mathbb{Q}) \cap \Z^n$ in $\mathbb{Q}^n=\Z^n \otimes_{\Z} \mathbb{Q}$.
\end{remark}
A wraith is \emph{torsion-free} if $\mathcal{R}$ is a full sublattice of $\Z^n$, or equivalently if the torsion group $T$ is trivial.\\

For example, crystallographic wraiths are torsion-free wraiths with $\Z^n = \Z\free{\mathcal{F}_{\Lambda}}$ and $\mathcal{R}= \langle \mathcal{L}_{\mathcal{F}_{\Lambda}} \rangle$.

\begin{remark}
Given a torsion-free wraith, the existence of an underlying Euclidean structure realizing it as a crystallographic wraith can be decided by linear programming.
\end{remark}

Two wraiths $(\Z^{n_1},\mathcal{R}_1)$ and $(\Z^{n_2},\mathcal{R}_2)$ are \emph{isomorphic} if  there exists an isometry $\phi : \Z^{n_1} \longrightarrow \Z^{n_2}$ such that $\phi(\mathcal{R}_1)=\mathcal{R}_2$.\\
In particular, this implies that $m=n$ and that the two wraiths have the same dimension.\\

This is compatible with the notion of isomorphism given for crystallographic wraiths.\\

We denote by $\pi=\pi_{(\Z^n,\mathcal{R})}$ the natural homomorphism from $\Z^n$ to $\Gamma$, and $\overline{\pi}=\overline{\pi_{(\Z^n,\mathcal{R})}}$ the morphism from $\Z^n$ to $\Gamma/2\Gamma$.\\

Observe that condition (3) is equivalent to the fact that the natural homomorphism $\overline{\pi}$ from $\Z^n$ onto $\Gamma/2\Gamma$ maps the standard basis $e_1,...,e_n$ of $\Z^n$ onto $n$ distincts nonzero elements of $\Gamma/2\Gamma$.

A feature of $\Gamma/2\Gamma$ is the following.
\begin{proposition}
\label{p31}
$\Gamma/2\Gamma$ is a vector space of dimension $d$ over the field $\mathbb{F}_2$ of 2 elements.
\end{proposition}
\begin{proof}
Let us remark, for the following proof, that since the group of torsion $T$ is of odd order, the multiplication by 2 is a bijection on $T$.
\begin{align*}
\Gamma /2 \Gamma &\simeq (\Z^d \times  T)/2(\Z^d \times T)\\
&\simeq (\Z^d \times  T)/(2\Z^d \times 2T)\\
&\simeq \Z^d/2\Z^d \times  T/2T\\
&\simeq \Z^d/2\Z^d \quad \text{since $T$ is of odd order}
\end{align*}
\end{proof}
It follows from Proposition \ref{p31} that $\Gamma/2\Gamma\backslash\{0\}$ is the $(d-1)$-dimensional projective space over the field $\mathbb{F}_2$ of 2 elements.\\

\begin{proposition}
\label{p32}
If $r \in \{ \pm e_1 \pm e_2 \pm e_3\}$ is an element of (squared Euclidean) norm 3 in $\mathcal{R}$, then $\overline{\pi}(e_1)$, $\overline{\pi}(e_2)$ and $\overline{\pi}(e_3)$ form a projective line in $\Gamma/2\Gamma \backslash \{0\}$.\\
However, a projective line of $\Gamma/2\Gamma\backslash\{0\}$ is generally not of this form.\\
\end{proposition}
We leave the easy proof to the reader.\\

Proposition \ref{p32} suggests to define \emph{line relations} as elements of norm 3 in $\mathcal{R}$. We can also replace $\mathcal{R}$ by its set of line relations in the definition of wraiths.\\

The projective viewpoint developed above leads to the following geometric interpretation.

A $d$-dimensional wraith can be seen as a pair $(\mathcal{S},\mathcal{R})$ where $\mathcal{S}$ is a subset of the $(d-1)$-dimensional projective space over $\mathbb{F}_2$ and $\mathcal{R} \subset \Z\free{\mathcal{S}}$ is a subgroup of the standard Euclidean lattice $\Z\free{\mathcal{S}}$ with the orthonormal basis $\{[s_1],...,[s_n]\}_{s_i \in \mathcal{S}}$ with the following properties :\begin{enumerate}
\item $\mathcal{R}$ does not contain elements of (squared Euclidean) norm 1 and 2,
\item $\mathcal{R}$ is generated by all elements of the form $\pm [s_{i_1}]\pm [s_{i_2}]\pm [s_{i_3}]$ in $(\mathcal{R} \otimes_\Z \mathbb{Q}) \cap \Z\free{\mathcal{S}}$, which are associated to lines in the projective space,
\item The quotient group $\Gamma = \Z^n/\mathcal{R}$ is isomorphic to $\Z^d \times T$ where $T$ is a torsion group of odd order.\\
\end{enumerate}

There are finitely many wraiths in a given dimension. We give here an immediate upper bound for the number of wraiths.
\begin{proposition}
Denoting by $w_d$ the number of $d$-dimensional wraiths, we have $$w_d \le 2^{d+2^{3d}}.$$ 
\end{proposition}
\begin{proof}
We start by noting that, thanks to the previous Propositions, if $(\Z^n,\mathcal{R})$ is a $d$-dimensional wraith, then $n$ is an integer in $[d,2^d-1]$.\\
Now, given such a $n$, we want to find an upper bound on the number of centrally symmetric subsets of $\Z^n$ containing only elements of (squared Euclidean) norm 3. There are $4\binom{n}{3} \leq n^3$ opposite pairs of vectors of (squared Euclidean) norm 3 in $\Z^n$.\\
Thus, the number of sets of pairs of vectors of (squared Euclidean) norm 3 in $\Z^n$ is $2^{4\binom{n}{3}} \leq 2^{n^3}<2^{2^{3d}}$ since $n<2^d$.\\
Therefore, an upper bound for the number of $d$-dimensional wraiths is $\som{n=d}{2^d-1}{2^{2^{3d}}}<2^{d+2^{3d}}$.\\
\end{proof}

A small refinement of the previous proposition is given by the following theorem.
\begin{theorem}
We have $$w_d \le 2^{2^d-1} \cdot 5^{\frac{1}{3}\binom{2^d-1}{2}}.$$
\end{theorem}
\label{th2}
\begin{proof}
To obtain this upper bound, we use the pairs given by geometric reformulation of the definition of $d$-dimensional wraiths above and we now proceed by a counting argument.\\
The first element $\mathcal{S}$ of the pair $(\mathcal{S},\mathcal{R})$ is a subset of the $(d-1)$-dimensional projective space over the field $\mathbb{F}_2$.\\
Since this projective space has $2^d-1$ elements, we bound the number of subsets associated to wraiths by the number $2^{2^d-1}$ of all subsets.
Then, once the first element $\mathcal{S}$ is fixed, the second element of the pair is a sublattice of $\Z\free{\mathcal{S}}$ generated by vectors of norm 3, which are associated to lines in the projective space, up to signs, with all three points contained in the subset $\mathcal{S}$.\\
Let us remark that one shows easily that a line in the projective space is associated to at most one pair of opposite line relations. Moreover, assuming that a line corresponds in fact to a pair of opposite line relations of a subset, there are four possible choices of signs for the coefficients of this pair of line relation.\\
Hence, for every projective line involving three elements of $\mathcal{S}$, we have to choose between five options for each line of the projective space : either we don't take any line relation associated, or we take a pair and we have 4 possible choices for signs, up to multiplication by $-1$.\\
Since there are $\frac{1}{3}{\binom{2^d-1}{2}}$ lines in the projective space of dimension $d-1$ over $\mathbb{F}_2$, this gives $5^{\frac{1}{3}\binom{2^d-1}{2}}$ possible subsets.
\end{proof}

The next section is an example of application of the framework developed in the last two sections to prove the classical result of finiteness of the number of perfect lattices.

\section{An application to perfect lattices}
\label{S4}
\leavevmode\\
In geometry of numbers, the notion of perfection was introduced by Voronoï for positive quadratic forms (\cite{Vor1}). We use the associated notion for Euclidean lattices.\\
A Euclidean lattice $(\Lambda,\scalar)$ is said to be \emph{perfect} if $\scalar$ is determined by its arithmetical minimum $\underset{v \in \Lambda \setminus \{0\}}{min}{\langle v , v \rangle}$ and its set of minimal vectors $\mathcal{M}_{\Lambda}$.\\
Equivalently, a lattice $(\Lambda,\scalar)$ is perfect if the set $\{x \otimes x,  x \in \mathcal{M}_{\Lambda}\}$ spans the $\binom{d+1}{2}$-dimensional space $Sym^2(\mathbb{E})$  of all symmetric elements in $\mathbb{E} \otimes \mathbb{E}$.\\
Such lattices are only considered up to similarity, defined as follows.
Two lattices $(\Lambda_1,\scalar_1)$ and $(\Lambda_2,\scalar_2)$ are \emph{similar} if there exists a surjective similarity $\mu : \Lambda_1 \longrightarrow \Lambda_2$ such that $\mu(\mathcal{M}_1)=\mathcal{M}_2$.\\
With the second definition of perfect lattices, one can also show that it is sufficient to prove that there exist $\mathcal{G}_1$ and $\mathcal{G}_2$ two subsets of minimal vectors of $\Lambda_1$ and $\Lambda_2$ respectively, such that $\{x \otimes x,  x \in \mathcal{G}_i\}$ spans  $Sym^2(\Lambda_i \otimes_{\Z} \R)$ and an isomorphism $\mu : \Lambda_1 \longrightarrow \Lambda_2$ such that $\mu(\mathcal{G}_1)=\mathcal{G}_2$.
The finiteness of the number of perfect lattices up to similarity (or equivalently the number of perfect forms up to equivalence and scalar multiplication) is a classical result of geometry of numbers. It has already been proved by Voronoï in \cite{Vor1}, and later by other authors such as Martinet (\cite{JM}), Schurmann (\cite{AS}), Bacher (\cite{Bac}), van Woerden (\cite{VW}) using other approaches.
Our goal is to prove this result using a point of view analogous to what has been explained in Section \ref{S2}.\\
To this end, notice that minimal vectors are facet vectors. This has indirectly been demonstrated by Voronoï in \cite{Vor1}. 
Now, we can prove the following Proposition.

\begin{proposition}
In a fixed dimension, there exist only finitely many perfect lattices up to similarity.
\end{proposition}

\begin{proof}
Let $\Lambda_1$ be a perfect lattice, with $\mathcal{F}_1$ a set of representatives of all pairs of facet vectors. We also denote $\mathcal{G}_1 \subset \mathcal{F}_1$ a set of minimal vectors such that $\{x \otimes x,  x \in \mathcal{G}_1\}$ spans  $Sym^2(\Lambda_1 \otimes_{\Z} \R)$. Using the notations of Section \ref{S2}, we associate to $\Lambda_1$ the triplet $(\mathcal{F}_1, \mathcal{L}_{\mathcal{F}_1},\mathcal{G}_1)$. We attach to these triplets a notion of equivalence analogous to the one developed in Section \ref{S2} :\\
Given two perfect lattices $\Lambda_1$ and $\Lambda_2$, we will say that their associated triplets $(\mathcal{F}_1, \mathcal{L}_{\mathcal{F}_1},\mathcal{G}_1)$ and $(\mathcal{F}_2, \mathcal{L}_{\mathcal{F}_2},\mathcal{G}_2)$ are \emph{isomorphic} if there exists an isomorphism $\rho : \Z\free{\mathcal{F}_1} \longrightarrow \Z\free{\mathcal{F}_2}$ such that :
\begin{itemize}
\item $\rho(\mathcal{F}_1)=\mathcal{F}_2$,
\item $\rho(\pm \mathcal{L}_{\mathcal{F}_1})=\pm \mathcal{L}_{\mathcal{F}_2}$,
\item $\rho(\pm \mathcal{G}_1)=\pm \mathcal{G}_2$.
\end{itemize}
Let us now fix $\Lambda_1$ and $\Lambda_2$ two perfect lattices with associated triplets $(\mathcal{F}_1, \mathcal{L}_{\mathcal{F}_1},\mathcal{G}_1)$ and $(\mathcal{F}_2, \mathcal{L}_{\mathcal{F}_2},\mathcal{G}_2)$ being equivalent, with $\rho : \Z \free{\mathcal{F}_1} \longrightarrow \Z \free{\mathcal{F}_2} $ the corresponding isomorphism.\\
Then $\pi_{\mathcal{F}_2} \circ \rho$ is surjective, with kernel $\ker(\pi_{\mathcal{F}_1})$. The first isomorphism theorem applied to $\pi_{\mathcal{F}_2} \circ \rho$ and $\pi_{\mathcal{F}_1}$ gives respectively an isomorphism $\overline{\pi_{\mathcal{F}_2} \circ \rho} : \Z \free{\mathcal{F}_1}/\ker(\pi_{\mathcal{F}_1}) \longrightarrow \Lambda_2$ and $\overline{\pi_{\mathcal{F}_1}}: \Z\free{\mathcal{F}_1}/\ker(\pi_{\mathcal{F}_1}) \longrightarrow \Lambda_1 $. Then, $\eta = \overline{\pi_{\mathcal{F}_2} \circ \rho} \circ \overline{\pi_{\mathcal{F}_1}}^{-1}$ is an isomorphism, verifying $\eta(\mathcal{G}_1)=\mathcal{G}_2$.\\
Thus, $\Lambda_1$ and $\Lambda_2$ are similar.\\
Now, let us demonstrate that the number of such triplets, up to equivalence, is finite.\\
The first two elements of these triplets can be seen as wraiths, and thus there exists a finite number of them up to isomorphism as demonstrated in Section \ref{S3}.\\
Finiteness for the third element follows from the fact that it is a subset of facet vectors corresponding to points of the $(d-1)$-dimensional projective space over $\mathbb{F}_2$.\\
\end{proof}

In addition to the proof of finiteness, Bacher and van Woerden found bounds for this number. With our method, we also obtain a bound, which is far from competing with the previous ones found, but has the advantage of simplicity.

\begin{corollary}
If $p_d$ denotes the number of perfect lattices up to similarity in dimension $d$, then
$$p_d\leq 2^{2^d-1} \cdot 5^{\frac{1}{3}\binom{2^d-1}{2}} \cdot \binom{2^d-1}{\binom{d+1}{2}}.$$
\end{corollary}
\begin{proof}
To obtain this upper bound, we use the triplets associated to perfect lattices described previously and we now proceed by a counting argument.\\
Using Theorem \ref{th2}, we have an upper bound for the number of pairs $(\mathcal{F},\mathcal{L}_{\mathcal{F}})$ corresponding to the first and second elements of the triplet $(\mathcal{F},\mathcal{L}_{\mathcal{F}},\mathcal{G})$, this upper bound being $2^{2^d-1} \cdot 5^{\frac{1}{3}\binom{2^d-1}{2}}$.
For $\mathcal{G}$, the third element of the triplet, let us recall that it must verify that $\{x \otimes x,  x \in \mathcal{G}\}$ generates the $\binom{d+1}{2}$-dimensional vector space $Sym^2(\mathbb{E})$ of $\mathbb{E} \otimes \mathbb{E}$. Thus, this last set contains a basis $\{v_1 \otimes v_1,..., v_{\binom{d+1}{2}} \otimes v_{\binom{d+1}{2}} \}$ of $Sym^2(\mathbb{E})$. Hence, $\mathcal{G}'=\{v_1 ,..., v_{\binom{d+1}{2}}\} \subset \mathcal{G}$ verifies that  $\{x \otimes x,  x \in \mathcal{G}'\}$ generates $Sym^2(\mathbb{E})$. The number of such subsets can be then bounded by the number of subsets of $\binom{d+1}{2}$ elements among all minimal vectors which are in particular facet vectors.\\
Since the number of pairs of such vectors is bounded by $2^d-1$ and $\left(\displaystyle\binom{n}{\binom{d+1}{2}}\right)_n$ is an increasing sequence, the number of all such subsets is bounded by $\displaystyle\binom{2^d-1}{\binom{d+1}{2}}$.\\
\end{proof}

This proof was an application of our framework to perfect lattices. But our point of view is valid for all lattices, and admits several other applications. One of them, which will be detailled elsewhere, is an algorithm to find all crystallographic wraiths in a given dimension, allowing a combinatorial classification of lattices.

\section*{Acknowledgment}

I would like to sincerely thank my supervisor, Roland Bacher, for his valuable advice and his support throughout this project.
\newpage
\bibliographystyle{alpha}
\bibliography{BiblioThesis.bib}

\end{document}